\documentclass[12pt, a4paper, leqno]{article}
\usepackage[margin=2.3cm]{geometry}
\usepackage[utf8]{inputenc} 
\usepackage{lmodern}
\usepackage[T1]{fontenc}
\usepackage[english]{babel}

\usepackage[runin]{abstract}
\usepackage{titling}

\usepackage{amsmath}
\usepackage{amsthm}
\usepackage{amsfonts}
\usepackage{amssymb}
\usepackage{mathtools}

\usepackage{enumitem}

\usepackage{cite}

\abslabeldelim{.}

\newcommand{\keywordsname}{Key words}
\makeatletter
\newcommand{\keywords}[1]{%
\def\thekeywords{#1}%
\begin{@bstr@ctlist}
\hspace*{\abstitleskip}{\abstractnamefont\keywordsname\@bslabeldelim}\abstracttextfont\
#1%
\par\end{@bstr@ctlist}
}
\makeatother

\newcommand{\subjclassname}{Mathematics subject classification}
\makeatletter
\newcommand{\subjclass}[2][2020]{%
\begin{@bstr@ctlist}
\hspace*{\abstitleskip}{\abstractnamefont\subjclassname\ (#1)\@bslabeldelim}\abstracttextfont\
#2%
\par\end{@bstr@ctlist}
}
\makeatother

\makeatletter
\def\and{
	\end{tabular}%
	and%
	\begin{tabular}[t]{c}}%
\makeatother

\makeatletter
\let\addresses\@empty      
\newcommand{\address}[2][]{\g@addto@macro\addresses{\address{#1}{#2}}}
\newcommand{\curraddr}[2][]{\g@addto@macro\addresses{\curraddr{#1}{#2}}}
\newcommand{\email}[2][]{\g@addto@macro\addresses{\email{#1}{#2}}}
\newcommand{\urladdr}[2][]{\g@addto@macro\addresses{\urladdr{#1}{#2}}}
\def\enddoc@text{
  \ifx\@empty\addresses \else\@setaddresses\fi}
\AtEndDocument{\enddoc@text}
\def\emailaddrname{E-mail address}
\def\@setaddresses{\par
  \nobreak \begingroup
  \interlinepenalty\@M
  \def\address##1##2{\begingroup%
    \par\addvspace\bigskipamount
    \@ifnotempty{##1}{(\ignorespaces##1\unskip) }%
    {\noindent\ignorespaces##2}\par\endgroup}%
  \def\email##1##2{\begingroup
    \@ifnotempty{##2}{\nobreak\noindent\emailaddrname
      \@ifnotempty{##1}{, \ignorespaces##1\unskip}\/:\space
      \ttfamily##2\par}\endgroup}%
  \addresses
  \endgroup
}

\makeatother

\newlist{conditions}{enumerate}{1}
\newlist{exconditions}{enumerate}{1}
\newlist{iconditions}{enumerate}{1}
\newlist{inthm}{enumerate}{1}
\setlist[conditions]{label=\normalfont(\alph*),ref=\normalfont(\alph*)}
\setlist[exconditions]{label=\normalfont(\roman*),ref=\normalfont(\roman*),wide,labelindent=0pt}
\setlist[iconditions]{label=\normalfont(\roman*),ref=\normalfont(\roman*)}
\setlist[inthm]{label=\normalfont(\thetheorem.\arabic*),ref=\normalfont(\thetheorem.\arabic*),wide,labelindent=0pt}

\mathchardef\mhyphen="2D

\newcommand{\CB}{\mathbb{C}}
\newcommand{\F}{\mathbb{F}}
\newcommand{\G}{\mathbb{G}}
\newcommand{\HB}{\mathbb{H}}
\newcommand{\PB}{\mathbb{P}}
\newcommand{\R}{\mathbb{R}}
\newcommand{\SB}{\mathbb{S}}
\newcommand{\T}{\mathbb{T}}
\newcommand{\Z}{\mathbb{Z}}

\newcommand{\C}{\mathcal{C}}

\newcommand{\U}{\mathcal{U}}
\newcommand{\VC}{\mathcal{V}}
\newcommand{\XC}{\mathcal{X}}
\newcommand{\YC}{\mathcal{Y}}

\newcommand{\Cinfty}{\C^{\infty}}

\newcommand{\GL}{\mathrm{GL}}
\newcommand{\ON}{\mathrm{O}}

\DeclareMathOperator{\dist}{dist}
\DeclareMathOperator{\Hom}{Hom}
\DeclareMathOperator{\Ker}{Ker}

\newtheorem{theorem}{Theorem}[section]
\newtheorem{corollary}[theorem]{Corollary}
\newtheorem{lemma}[theorem]{Lemma}
\newtheorem{proposition}[theorem]{Proposition}
\theoremstyle{definition}
\newtheorem*{acknowledgements}{Acknowledgements}
\newtheorem{definition}[theorem]{Definition}
\newtheorem{example}[theorem]{Example}
\newtheorem{notation}[theorem]{Notation}

\DeclarePairedDelimiter\abs{\lvert}{\rvert}%
\DeclarePairedDelimiter\norm{\lVert}{\rVert}%

\makeatletter
\let\oldabs\abs
\def\abs{\@ifstar{\oldabs}{\oldabs*}}
\let\oldnorm\norm
\def\norm{\@ifstar{\oldnorm}{\oldnorm*}}
\makeatother

\title{Approximation and homotopy in regulous geometry}
\date{}
\author{Wojciech Kucharz}

\address{Wojciech Kucharz\\Institute of Mathematics\\Faculty of Mathematics and Computer
Science\\Jagiellonian University\\\L{}ojasiewicza 6\\30-348
Krak\'ow\\Poland}
\email{Wojciech.Kucharz@im.uj.edu.pl}

\usepackage[pdftex, pdfauthor={\theauthor}, pdftitle={\thetitle}]{hyperref}

\begin{document}
\maketitle
\thispagestyle{empty}

\begin{abstract}
Let $X$, $Y$ be nonsingular real algebraic sets. A map $\varphi \colon X \to Y$ is said to be $k$-regulous,
where $k$ is a nonnegative integer, if it is of class $\C^k$ and the restriction of $\varphi$ to some Zariski open
dense subset of $X$ is a regular map. Assuming that $Y$ is uniformly rational, and $k \geq 1$, we prove that a $\Cinfty$ map $f \colon X \to Y$ can be approximated
by $k$-regulous maps in the $\C^k$ topology if and only if $f$ is homotopic to a $k$-regulous map. The class of uniformly rational real algebraic varieties includes spheres, Grassmannians and rational nonsingular surfaces, and is stable under blowing up nonsingular centers. 
Furthermore, taking $Y=\SB^p$ (the unit $p$-dimensional sphere), we obtain several new results
on approximation of $\Cinfty$ maps from $X$ into~$\SB^p$ by $k$-regulous maps in the $\C^k$ topology, for $k \geq 0$.
\end{abstract}


\hypersetup{pdfkeywords={\thekeywords}}

\section{Introduction}\label{sec:1}
Regulous geometry has recently emerged as a subfield of real algebraic geometry. It deals with rational maps that admit continuous extensions or extensions satisfying certain differentiability conditions. In the following, we develop new methods that lead to a much better understanding of the relationship between the concepts of approximation and homotopy of maps in the framework of regulous geometry.

Throughout this paper we use the term \emph{real algebraic variety} to
mean a ringed space with structure sheaf of $\R$-algebras of $\R$-valued
functions, which is isomorphic to a Zariski locally closed subset of
real projective $n$-space $\PB^n(\R)$, for some $n$, endowed with the
Zariski topology and the sheaf of regular functions. This is compatible
with \cite{bib4, bib53B}, which contain a detailed exposition of real algebraic
geometry. Recall that each real algebraic variety in the sense used here
is actually affine, that is, isomorphic to an algebraic subset of $\R^n$
for some $n$, see \cite[Proposition~3.2.10 and Theorem~3.4.4]{bib4}.
Morphisms of real algebraic varieties are called \emph{regular maps}.
Each real algebraic variety carries also the Euclidean topology
determined by the usual metric on $\R$. Unless explicitly stated
otherwise, all topological notions relating to real algebraic varieties
refer to the Euclidean topology.

As a matter of convention, all $\Cinfty$ manifolds will be Hausdorff and
second countable. The space $\C^k(M,N)$ of $\C^k$ maps between $\Cinfty$
manifolds, where $k$ is a nonnegative integer or $k=\infty$, is endowed
with the $\C^k$ topology (see \cite[pp.~34, 36]{bib27} or
\cite[p.~311]{bib53} for the definition of this topology and note that in
\cite{bib27} it is called the weak $\C^k$ topology; the $\C^0$ topology
is just the compact-open topology).

Let $X$, $Y$ be two nonsingular real algebraic varieties. A map $f
\colon X \to Y$ is said to be \emph{regulous} if it is continuous on~$X$
and there exists a Zariski open dense subset~$U$ of~$X$ such that the
restriction $f|_U \colon U \to Y$ is a regular map. Let $X(f)$ denote
the union of all such~$U$. The complement $P(f) \coloneqq X \setminus
X(f)$ of $X(f)$ is the smallest Zariski closed subset of~$X$ for which
the restriction $f|_{X \setminus P(f)} \colon X \setminus P(f) \to Y$ is
a regular map. If $f(P(f)) \neq Y$, we say that $f$ is a \emph{nice}
regulous map. In the literature regulous maps are also called \emph{continuous rational maps} \cite{bib31, bib32, bib34, bib36, bib37, bib38, bib40, bib43, bib44} or \emph{stratified-regular maps} \cite{bib39, bib46, bib54}. The concise name ``regulous'' was coined by Fichou, Huisman, Mangolte and Monnier \cite{bib17}. Since the publication of \cite{bib34} in 2009 several mathematicians have devoted their attention to regulous maps, see \cite{bib3, bib15,
bib16, bib17, bib18, bib19, bib20, 
bib21, bib31, bib32, bib34, bib36,
bib37, bib38, bib39, bib40, bib41, bib42, bib43, bib44, bib45, bib46,
bib47, bib50, bib54, bib55} and the references therein.

A map $f \colon X \to Y$ is said to be
\emph{$k$-regulous}, where $k$ is a nonnegative integer or $k = \infty$,
if it is both regulous and of class $\C^k$. Thus, less formally, a $k$-regulous map is a $\C^k$ map that admits a rational representation. Obviously, ``$0$-regulous''
is the same as ``regulous''. As observed in
\cite[Proposition~2.1]{bib34}, $\infty$-regulous maps coincide with
regular maps, and these are usually studied separately. A standard example of a $k$-regulous function, with $k$ a~nonnegative integer, is $f \colon \R^2 \to \R$ defined by
\begin{equation*}
    f(x,y) = \frac{x^{3+k}}{x^2+y^2} \quad \text{for } (x,y) \neq (0,0) \quad \text{and} \quad f(0,0)=0.
\end{equation*}
Clearly, $f$ is not of class $\C^{k+1}$.

We say that a $\C^l$ map $f \colon X \to Y$ can be \emph{approximated by
$k$-regulous maps in the $\C^k$ topology}, where $0 \leq k \leq l \leq \infty$, if for every neighborhood $\U$
of $f$ in $\C^k(X,Y)$ there exists a $k$-regulous map that belongs to
$\U$. Investigating whether or not the map $f$ admits approximation by
$k$-regulous maps in the $\C^k$ topology, we may assume without loss of
generality that $f$~is of class $\Cinfty$. This is justified since the
set $\Cinfty(X,Y)$ is dense in the space $\C^k(X,Y)$.

\begin{definition}\label{def-1-1}
An $n$-dimensional real algebraic variety $Y$ is said to be \emph{uniformly rational} if every point in $Y$ has a Zariski open neighborhood that is biregularly isomorphic to a Zariski open subset of $\R^n$.
\end{definition}

Clearly, every uniformly rational real algebraic variety is nonsingular of pure dimension. An intriguing open question posed by Gromov is whether every rational nonsingular variety is uniformly rational, see \cite[p.~885]{bib26} and \cite{bib14A} for the discussion involving complex varieties.

One of our main results is the following.

\begin{theorem}\label{th-1-2}
Let $k$ be a positive integer, $X$ a nonsingular real algebraic
variety, and $Y$ a uniformly rational real algebraic variety. Then, for a $\Cinfty$ map ${f \colon X \to Y}$, the following
conditions are equivalent:
\begin{conditions}
\item\label{th-1-2-a} $f$ can be approximated by $k$-regulous maps in
the $\C^k$ topology.

\item\label{th-1-2-b} $f$ is homotopic to a $k$-regulous map.
\end{conditions}
\end{theorem}

It is an open question whether Theorem~\ref{th-1-2} holds for $k=0$ or $k=\infty$. We also have a more general result, Theorem~\ref{th-4-2}, in which the target variety $Y$ need not be rational.

The following example illustrates the scope of applicability of Theorem~\ref{th-1-2}.\goodbreak

\begin{example}\label{ex-1-3}
Here are some uniformly rational real algebraic varieties.

\begin{exconditions}[widest=iii]
\item\label{ex-1-3-i} For any nonnegative integer~$n$, the unit $n$-sphere
\begin{equation*}
\SB^n \coloneqq \{ (x_0,\ldots,x_n) \in \R^{n+1} : x_0^2 + \cdots+ x_n^2
= 1 \}
\end{equation*}
is uniformly rational because $\SB^n$ with one point removed is biregularly isomorphic to~$\R^n$.

\item\label{ex-1-3-ii} Let $\F$ stand for $\R$, $\CB$ or $\HB$, where $\HB$ is the (skew) field of quaternions. The Grassmannian $\G_d(\F^n)$ of $d$-dimensional $\F$-vector subspaces of $\F^n$ can be regarded as a real algebraic variety (see \cite[pp.~72, 73, 352]{bib4}) and as such is uniformly rational.

\item\label{ex-1-3-iii} Rational nonsingular real algebraic surfaces are uniformly rational. As detailed in \cite[Subsection~2.2]{bib53A}, this follows from Comessatti's theorem \cite[p.~257]{bib17A}, whose modern proofs are given in \cite[Theorem~30]{bib33A} and \cite[p.~137, Proposition~6.4]{bib56A}.

\item\label{ex-1-3-iv} Blow-ups with nonsingular centers of uniformly rational varieties remain uniformly rational, and the proof given in \cite[p.~885]{bib26} and \cite{bib14A} in a complex setting also works for real algebraic varieties.
\end{exconditions}
\end{example}

All previous results on approximation by $k$-regulous maps concern maps with values in Grassmann varieties \cite{bib34, bib45, bib46, bib54} or unit spheres \cite{bib3, bib34, bib36, bib37, bib40, bib42, bib43, bib46, bib47, bib55}. Theorem~\ref{th-1-2} does not provide any new information in the former case (at least for $X$ compact), but opens up new possibilities in the latter.

In view of Theorem~\ref{th-1-2} and Example~\ref{ex-1-3}\ref{ex-1-3-i}, we get immediately the following result on maps into $\SB^p$.

\begin{corollary}\label{cor-1-4}
Let $k$ be a positive integer, $X$ a nonsingular
real algebraic variety, and $p$
a nonnegative integer. Then, for a $\Cinfty$ map
$f \colon X \to \SB^p$, the following conditions are equivalent:
\begin{conditions}
\item\label{cor-1-4-a} $f$ can be approximated by $k$-regulous maps in
the $\C^k$ topology.

\item\label{cor-1-4-b} $f$ is homotopic to a $k$-regulous map.\qed
\end{conditions}
\end{corollary}

 Up to now, Corollary~\ref{cor-1-4} with $X$ compact and $\dim X \geq p \geq 1$   has only been known for three special values
of~$p$, namely, $p=1,2$ or $4$ \cite[Corollary~3.8]{bib34}. Since $\infty$-regulous is the same as regular, the value $k=\infty$ is allowed in Corollary~\ref{cor-1-4} according to \cite[Corollary~1.2]{bib13}.

In Theorem~\ref{th-1-2} and Corollary~\ref{cor-1-4} the
integer~$k$ is assumed to be positive, that is, the case $k=0$ is
excluded (which perhaps is not necessary). However, we have the
following criterion involving nice regulous maps, which are $0$-regulous
by definition.

\begin{corollary}\label{cor-1-5}
Let $X$ be a compact nonsingular real algebraic variety and let $k$, $p$
be two nonnegative integers. Assume that a $\Cinfty$ map $f \colon X \to
\SB^p$ is homotopic to a nice regulous map. Then $f$ can be approximated
by $k$-regulous maps in the $\C^k$ topology.
\end{corollary}

\begin{proof}
To deal with the case $k=0$ we choose an integer $l > k$. Since $f$ is
homotopic to a nice regulous map, it is also homotopic to a nice
$l$-regulous map \cite[Theorem~2.4]{bib34}. Therefore, by
Corollary~\ref{cor-1-4}, $f$ can be approximated by $l$-regulous maps in
the $\C^l$ topology. The conclusion follows.
\end{proof}

In connection with Corollary~\ref{cor-1-5}, it is natural to raise the
question whether every regulous map from $X$ into $\SB^p$ is homotopic
to a nice regulous map. According to \cite[Theorem~2.4]{bib34}, the
continuous maps into unit spheres that are homotopic to nice regulous
maps are characterized in terms of framed cobordism classes via the
Pontryagin--Thom construction. Next we focus on approximation by nice
$k$-regulous maps.

Let $X$ be a nonsingular real algebraic variety. If $Z$~is a nonsingular
Zariski locally closed subset of~$X$, then its Zariski closure~$V$
in~$X$ is of the form $V=Z \cup W$, where $W$~is a Zariski closed subset
of~$X$ with $Z \cap W = \varnothing$ and $\dim W < \dim Z$. Clearly,
$Z$~is precisely the nonsingular locus of~$V$, assuming that $Z$~is
closed in~$ X$ (in the Euclidean topology). An illustrative example is
provided by $Z=C \setminus \{(0:0:1)\}$, where $C$~is the singular cubic
curve
\begin{equation*}
C \coloneqq \{ (x:y:z) \in \PB^2(\R) : y^2z - x^3 - x^2z = 0 \}
\end{equation*}
in the real projective plane $\PB^2(\R)$.\goodbreak

A compact $\Cinfty$ submanifold~$M$ of~$X$ is said to admit a \emph{weak
algebraic approximation} if, for every neighborhood $\U$ of the
inclusion map $M \hookrightarrow X$ in the space $\Cinfty(M,X)$, there
exists a $\Cinfty$ embedding $e \colon M \to X$ in~$\U$ such that $e(M)$
is a nonsingular Zariski locally closed subset of~$X$.

Assume that $X$~is compact. A $\Cinfty$ map $f \colon X \to \SB^p$ is
said to be \emph{adapted} (resp. \emph{weakly adapted}) if there exists
a regular value $y \in \SB^p$ for $f$ such that $f^{-1}(y)$ is a
nonsingular Zariski locally closed subset of~$X$ (resp. the $\Cinfty$
submanifold $f^{-1}(y)$ of~$X$ admits a weak algebraic approximation).

Our main result on approximation of $\Cinfty$ maps into unit spheres by
nice $k$-regulous maps is the following.

\begin{theorem}\label{th-1-6}
Let $X$ be a compact nonsingular real algebraic variety and let $k$, $p$
be two integers, with $k \geq 0$, $p \geq 1$. Then, for a $\Cinfty$ map
$f \colon X \to \SB^p$, the following conditions are equivalent:
\begin{conditions}
\item\label{th-1-6-a} $f$ can be approximated by nice $k$-regulous maps
in the $\C^k$ topology.

\item\label{th-1-6-b} $f$ can be approximated by adapted $\Cinfty$ maps
in the $\C^k$ topology.

\item\label{th-1-6-c} $f$ can be approximated by weakly adapted $\Cinfty$ maps in
the $\C^k$ topology.
\end{conditions}
\end{theorem}

Using Theorem~\ref{th-1-6}, we can obtain two approximation results that do not require any technical assumptions.

\begin{corollary}\label{cor-1-7}
Let $X$ be a compact nonsingular real algebraic variety of dimension~$p$ and let $k$ be a nonnegative integer. Then every $\Cinfty$ map from $X$
into $\SB^p$ can be approximated by nice $k$-regulous maps in the $\C^k$
topology.
\end{corollary}

\begin{proof}
Since $\dim X = p$, every $\Cinfty$ map from $X$ into $\SB^p$ is adapted, and hence the conclusion follows from Theorem~\ref{th-1-6}.
\end{proof}

For maps between unit spheres we have the following.

\begin{theorem}\label{th-1-8}
Let $k$ be a nonnegative integer. Then, for every pair $(n,p)$ of
nonnegative integers, every $\Cinfty$ map from $\SB^n$ into $\SB^p$ can
be approximated by nice $k$-regulous maps in the $\C^k$ topology.
\end{theorem}

\begin{proof}
Let $M$ be a compact $\Cinfty$ submanifold of~$\SB^n$ with $\dim M < n$.
Let $a$ be a point in $\SB^n \setminus M$ and let $\rho \colon \SB^n
\setminus \{a\} \to \R^n$ be the stereographic projection. By
\cite[Theorem~A]{bib1}, the $\Cinfty$ submanifold $\rho(M)$ of~$\R^n$
admits a weak algebraic approximation. Since $\rho$~is a biregular
isomorphism, $M$~admits a weak algebraic approximation in~$\SB^n$.
Consequently, if $p\geq1$, then every $\Cinfty$ map $\SB^n \to \SB^p$ is
weakly adapted, and hence the conclusion follows by
Theorem~\ref{th-1-6}. The case $p=0$ is trivial.
\end{proof}

It remains undecided whether or not for any pair $(n,p)$ of nonnegative
integers every $\Cinfty$ map $\SB^n \to \SB^p$ can be approximated by
regular maps in the $\Cinfty$ (or $\C^0$) topology, see \cite{bib13} for more information.

We now turn to a different characterization of the $\Cinfty$ maps into
unit spheres that can be approximated by nice $k$-regulous maps.

Let $X$ be a compact nonsingular real algebraic variety and let $p$ be
an integer with $0 \leq p \leq n \coloneqq \dim X$. Following
\cite[p.~19]{bib43}, we say that a cohomology class $v \in H^p(X;
\Z/2)$ is \emph{adapted} if the homology class in $H_{n-p}(X; \Z/2)$
Poincar\'e dual to~$v$ can be represented by a compact
$(n-p)$-dimensional $\Cinfty$ submanifold~$Z$ of~$X$, embedded with
trivial normal bundle, such that $Z$~is a nonsingular Zariski locally
closed subset of~$X$. Denote by $A^p(X; \Z/2)$ the subgroup of $H^p(X;
\Z/2)$ generated by all adapted cohomology classes.

\begin{theorem}\label{th-1-9}
Let $X$ be a compact nonsingular real algebraic variety and let $k$, $p$
be two integers, with $k \geq 0$, $2p \geq \dim X + 1$. Then, for a
$\Cinfty$ map $f \colon X \to \SB^p$, the following conditions are
equivalent:
\begin{conditions}
\item\label{th-1-9-a} $f$ can be approximated by nice $k$-regulous maps
in the $\C^k$ topology.

\item\label{th-1-9-b} $f$ can be approximated by nice regulous maps in
the $\C^0$ topology.

\item\label{th-1-9-c} $f$ is homotopic to a nice regulous map.

\item\label{th-1-9-d} $f^*(\sigma_p) \in A^p(X; \Z/2)$, where $f^*
\colon H^p(\SB^p; \Z/2) \to H^p(X; \Z/2)$ is the induced homomorphism
and $\sigma_p$ is the unique nonzero element in $H^p(\SB^p;\Z/2) \cong
\Z/2$.
\end{conditions}
\end{theorem}

It is natural to wonder whether the assumption $2p \geq \dim X + 1$ in
Theorem~\ref{th-1-9} is necessary.

The following example sheds light on some relationships between regular,
$k$-regulous, and $\Cinfty$ maps with values in unit spheres.

\begin{example}\label{ex-1-10}
Let $k$ be a nonnegative integer and let $\T^n = \SB^1 \times \cdots
\times \SB^1$ be the $n$-fold product of $\SB^1$. One readily checks that
$A^p(\T^n; \Z/2) = H^p(\T^n; \Z/2)$ for $0 \leq p \leq n$. Hence, in view
of Theorem~\ref{th-1-9}, if $2p \geq n+1$, then every $\Cinfty$ map
$\T^n \to \SB^p$ can be approximated by nice $k$-regulous maps in the
$\C^k$ topology. On the other hand, by \cite[Theorem~3.2]{bib6}, if
$n$~is a positive even integer, then every regular map $\T^n \to \SB^n$
is null homotopic (of course there are $\Cinfty$ maps $\T^n \to \SB^n$
that are not null homotopic and they do not admit approximation by
regular maps in the $\C^0$ topology). In particular, we cannot allow
$k=\infty$ in Theorems~\ref{th-1-6},~\ref{th-1-9} and Corollaries
\ref{cor-1-5}, \ref{cor-1-7}. Furthermore, according to
\cite[Theorem~2.8]{bib37}, if $n > p \geq 1$, then there exist a
nonsingular real algebraic variety~$X$ and a $\Cinfty$ map $f \colon X
\to \SB^p$ such that $X$~is diffeomorphic to~$\T^n$ and $f$~is not
homotopic to any regulous map.
\end{example}

There is ample evidence that the phenomenon exhibited in
Example~\ref{ex-1-10} is quite common: $k$-regulous maps, where $k$ is a
nonnegative integer, are more
flexible than regular maps. Approximation by regular maps is
investigated in \cite{bib2a,bib13} and numerous earlier papers \cite{bib5,
bib7,
bib8,bib9,bib10,bib11,
bib12, bib23,
bib24,
bib25, bib28,
bib29,
bib30,bib33,bib35,bib49,bib51,bib52}.

Theorems~\ref{th-1-2}, \ref{th-1-6} and~\ref{th-1-9} are proved in
Section~\ref{sec:4}. The methods employed in the proof of Theorem~\ref{th-1-2} are developed in Sections \ref{sec:2} and \ref{sec:3}. The inspiration for these methods originates from complex geometry, especially Gromov's article \cite{bib26} and the related works of
Forstneri\v{c} and others elaborated in \cite{bib22}. Of independent interest are Theorems \ref{th-3-6}, \ref{th-3-7}, \ref{th-4-2} and \ref{th-4-3}, which are refined versions of Theorem~\ref{th-1-2}. We derive Theorem~\ref{th-1-6} by combining Corollary~\ref{cor-4-4} with some results of \cite{bib34}. For the proof of Theorem~\ref{th-1-9} essential are Theorem~\ref{th-1-6}, \cite{bib43} and Benoist's paper \cite{bib2}. The results on maps into unit spheres announced above are significant improvements upon \cite{bib37, bib40, bib43}, which deal exclusively with
approximation by nice regulous maps in the $\C^0$ topology.

\section{Malleability and local malleability properties}\label{sec:2}


As in \cite{bib46}, by a \emph{stratification} of a real algebraic variety~$V$ we mean a finite collection $\VC$ of pairwise disjoint Zariski locally closed subsets whose union is~$V$. Each element of $\VC$ is called a \emph{stratum}; a stratum can be empty.

\begin{definition}\label{def-2-1}
Let $k$ be a nonnegative integer or $k=\infty$, $X$ and $Y$ nonsingular real algebraic varieties, $\XC$ a stratification of $X$, and $\YC$ a stratification of $Y$.

 A map $f \colon X \to Y$ is said to be \emph{$(k,\XC)$-regular} if it is of class $\C^k$ and for each stratum $S \in \XC$ the restriction $f|_S \colon S \to Y$ is a regular map. If, in addition, $f(S)$ is contained in a stratum $T \in \YC$, then $f$ is said to be \emph{$(k,\XC,\YC)$-regular}.

\end{definition}

We are now in a position to give an alternative description of $k$-regulous maps (see also \cite[Propositon~8]{bib32} and \cite[Th\'eor\`eme~4.1]{bib17}).

\begin{lemma}\label{lem-2-2}
Let $k$, $X$, $Y$, $\XC$, $\YC$ be as in Definition~\ref{def-2-1}.
\begin{iconditions}
\item\label{lem-2-2-i} If a map $f \colon X \to Y$ is $(k,\XC)$-regular, then it is $k$-regulous.

\item\label{lem-2-2-ii} If a map $f \colon X \to Y$ is $k$-regulous, then there exists a stratification $\XC'$ of $X$ such that $X \setminus P(f)$ is  a stratum of $\XC'$ and $f$ is $(k, \XC')$-regular ($P(f)$ is the Zariski closed subset of~$X$ defined in Section~\ref{sec:1}).

\item\label{lem-2-2-iii} If a map $f \colon X \to Y$ is $k$-regulous, then there exists a stratification $\XC''$ of $X$ such that $f$ is $(k, \XC'', \YC)$-regular.
\end{iconditions}
\end{lemma}

\begin{proof}
The proof of \ref{lem-2-2-i} is straightforward, and \ref{lem-2-2-ii} follows from \cite[Proposition~8 and p.~91]{bib32} (\!\!\cite{bib32} deals with $Y=\R$, but the general case follows at once because $Y$ can be regarded as a subvariety of $\R^p$, for some $p$). To prove \ref{lem-2-2-iii}, we choose a stratification $\XC'$ as in \ref{lem-2-2-ii}, and define
\begin{equation*}
    \XC'' \coloneqq \{(f|_S)^{-1}(T) : S \in \XC' \text{ and } T \in \YC\}.\qedhere
\end{equation*}
\end{proof}

For the sake of clarity, we record the following (see also \cite[Corollaire~4.14]{bib17}).

\begin{corollary}\label{cor-2-3}
Let $X$, $Y$, $Z$ be nonsingular real algebraic varieties, and $k$ a nonnegative integer or $k=\infty$. Assume that $f \colon X \to Y$ and $g \colon Y \to Z$ are $k$-regulous maps. Then the composite map $g \circ f$ is also $k$-regulous.
\end{corollary}

\begin{proof}
By Lemma~\ref{lem-2-2}\ref{lem-2-2-ii}, there exists a stratification $\YC$ of $Y$ such that the map $g$ is $(k, \YC)$-regular. In view of Lemma~\ref{lem-2-2}\ref{lem-2-2-iii}, we can choose a stratification $\XC$ of $X$ such that the map $f$ is $(k, \XC, \YC)$-regular. Consequently, the map $g \circ f$ is $(k, \XC)$-regular, so it is $k$-regulous by Lemma~\ref{lem-2-2}\ref{lem-2-2-i}.
\end{proof}

In what follows we work with vector bundles, which are always $\R$-vector bundles. Let $Y$ be a real algebraic variety. Given a vector bundle $p \colon E \to Y$ over $Y$, with total space~$E$ and bundle projection $p$, we sometimes refer to $E$ as a vector bundle over $Y$. If $y$ is a point in~$Y$, we let $E_y \coloneqq p^{-1}(y)$ denote the fiber of $E$ over $y$ and write $0_y$ for the zero vector in~$E_y$. We call the set $Z(E) \coloneqq \{0_y \in E : y \in Y\}$ the zero section of $E$.

For the general theory of algebraic vector bundles over real algebraic varieties we refer the reader to \cite[Section~12.1]{bib4}. For each algebraic vector bundle $E$ over $Y$ there exist a nonnegative integer $n$  and a surjective algebraic morphism from the product vector bundle $Y \times \R^n$ onto $E$ \cite[Theorem~12.1.7]{bib4}.

Assuming that $Y$~is a nonsingular real algebraic variety, we write $TY$
for the tangent bundle to~$Y$, and $T_yY$ for the tangent space to~$Y$
at~$y \in Y$.

The following notions will be crucial in the proofs of all our main theorems.

\begin{definition}\label{def-2-4}
Let $Y$ be a nonsingular real algebraic variety, and $k$ a~positive integer or $k=\infty$.
\begin{iconditions}
\item\label{def-2-4-i} A \emph{$k$-regulous spray} for $Y$ is a triple
$(E, p, s)$, where $p \colon E \to Y$ is an algebraic vector bundle
over~$Y$ and $s \colon E \to Y$ is a $k$-regulous map 
such that $s(0_y) = y$ for all $y \in Y$.

\item\label{def-2-4-ii} A $k$-regulous spray $(E, p, s)$ for $Y$ is said
to be \emph{dominating} if the derivative 
\begin{equation*}
{d_{0_y}s \colon T_{0_y}E \to T_yY}
\end{equation*}
maps the subspace $E_y = T_{0_y}E_y$ of $T_{0_y}E$ onto $T_yY$, that is,
\begin{equation*}
d_{0_y}s(E_y) = T_y Y \quad \text{ for all } y \in Y.
\end{equation*}

\item\label{def-2-4-iii} The variety $Y$ is called \emph{$k$-malleable} if
it admits a dominating $k$-regulous spray.
\end{iconditions}
For simplicity, $\infty$-regulous sprays, dominating $\infty$-regulous sprays and $\infty$-malleable varieties are called \emph{sprays}, \emph{dominating sprays} and \emph{malleable varieties}, respectively.
\end{definition}

Since $\infty$-regulous maps are regular, it follows that the concepts of spray, dominating spray, and malleable variety in Definition~\ref{def-2-4} above are identical with those in \cite[Definition~2.1]{bib13}.

\begin{lemma}\label{lem-2-5}
Let $Y$ be a nonsingular real algebraic variety, and $k$ a positive integer or $k=\infty$. If the variety $Y$ is $k$-malleable, then it admits a dominating $k$-regulous spray $(E,p,s)$ such that $p \colon E = Y \times \R^n \to Y$ is the product vector bundle.
\end{lemma}

\begin{proof}
Let $(\tilde E, \tilde p, \tilde s)$ be a dominating $k$-regulous spray for $Y$. Choose a nonnegative integer~$n$ and a surjective algebraic morphism $\alpha \colon E \to \tilde E$ from the product vector bundle $p \colon E = Y \times \R^n \to Y$ onto $\tilde p \colon \tilde E \to Y$. By Corollary~\ref{cor-2-3}, the map $s \colon E \to Y$, $s(y,v) = \tilde s (\alpha(y,v))$ is $k$-regulous, so $(E,p,s)$ is a dominating $k$-regulous spray for $Y$.
\end{proof}

We now recall important examples of malleable real algebraic varieties.

\begin{example}\label{ex-2-6}
Let $G$ be a linear real algebraic group, that is, a Zariski closed subgroup of the general linear group $\GL_n(\R)$, for some $n$. A \emph{$G$-space} is a real algebraic variety $Y$ on which $G$ acts, the action $G \times Y \to Y$, $(a,y) \mapsto a \cdot y$ being a regular map. We say that a $G$-space $Y$ is \emph{good} if $Y$ is nonsingular and for every point $y \in Y$ the derivative of the map $G \to Y$, $a \mapsto a \cdot y$ at the identity element of $G$ is surjective. Clearly, if $Y$ is \emph{homogeneous} for~$G$ (that is, $G$ acts transitively on $Y$), then $Y$ is a good $G$-space. By \cite[Proposition~2.8]{bib13}, each good $G$-space is malleable.

In particular, the unit $n$-sphere $\SB^n$ and real projective $n$-space $\PB^n(\R)$ are malleable varieties, being homogeneous spaces for the orthogonal group $\ON(n+1) \subset\GL_{n+1}(\R)$.
\end{example}

It is also convenient to introduce the following definition.

\begin{definition}\label{def-2-7}
Let $Y$ be a nonsingular real algebraic variety.
\begin{iconditions}
\item\label{def-2-7-i} A \emph{local spray} for $Y$ is a regular map $\sigma \colon U \times \R^n \to Y$, where $U$ is a Zariski open subset of $Y$ and $n$ is a nonnegative integer, such that $\sigma(y,0) = y$ for all $y \in U$.

\item\label{def-2-7-ii} A local spray $\sigma \colon U \times \R^n \to Y$ for $Y$ is said to be \emph{dominating} if for every point $y \in U$ the derivative of the map $\sigma(y, \cdot) \colon \R^n \to Y$, $v \mapsto \sigma(y,v)$ at $0 \in \R^n$ is surjective.

\item\label{def-2-7-iii} The variety $Y$ is called \emph{locally malleable} if for every point $p \in Y$ there exists a dominating local spray $\sigma \colon U \times \R^n \to Y$ for $Y$ with $p \in U$.
\end{iconditions}
\end{definition}

It follows directly from Definitions \ref{def-2-4} and \ref{def-2-7} that each malleable real algebraic variety is locally malleable. It is plausible that the converse also holds, but we can only prove the following weaker result.

\begin{proposition}\label{prop-2-8}
Let $Y$ be a locally malleable nonsingular real algebraic variety. Then, for each positive integer $k$, the variety $Y$ is $k$-malleable.
\end{proposition}

\begin{proof}
Since the variety~$Y$ is quasi-compact in the Zariski topology, there exists a finite collection
%
    $\{\sigma_i \colon U_i \times \R^{n_i} \to Y : i=1,\ldots,q\}$ 
%
of dominating local sprays for $Y$ such that the Zariski open sets $U_i$ are nonempty and cover~$Y$. Choose a regular function $\varphi_i \colon Y \to \R$  with
$\varphi_i^{-1}(0) = Y \setminus U_i$ for $i=1,\ldots,q$.

Let $k$ be a positive integer. By Lemma~\ref{lem-2-9} below, there exists a
positive integer $r$ such that for 
$i=1,\ldots,q$ the map $\sigma_i^{(r)} \colon Y
\times \R^{n_i} \to Y$ defined by
\begin{equation*}
    \sigma_i^{(r)}(y,v_i) =
    \begin{cases}
    \sigma_i(y, \varphi_i(y)^r v_i) & \text{for } (y,v_i) \in U_i \times \R^{n_i}\\
    y & \text{for } (y,v_i) \in (Y \setminus U_i) \times \R^{n_i}
    \end{cases}
\end{equation*}
is $k$-regulous. Moreover, by construction, for every point $y \in U_i$ the derivative of the map $\sigma_i^{(r)}(y,\cdot) \colon \R^{n_i} \to Y$, $v_i \mapsto \sigma_i^{(r)}(y,v_i)$ at $0 \in \R^{n_i}$ is surjective. For $i=1,\ldots,q$, we define recursively a map $s_i \colon Y
\times \R^{n_1} \times \cdots \times \R^{n_i} \to Y$ by
\begin{align*}
&s_i = \sigma_1^{(r)} & &\text{for } i=1,\\
&s_i(y, v_1, \ldots, v_{i-1}, v_i) = \sigma_i^{(r)} (s_{i-1}(y, v_1,
\ldots, v_{i-1}), v_i) & &\text{for } i \geq 2.
\end{align*}
By Corollary~\ref{cor-2-3}, the maps $s_i$ are $k$-regulous. We obtain a dominating $k$-regulous spray $(E, p, s)$ for $Y$, where
\begin{equation*}
p \colon E = Y \times \R^{n_1} \times \cdots \times \R^{n_q} \to Y
\end{equation*}
is the product vector bundle and $s = s_q$. Thus the variety~$Y$ is $k$-malleable.
\end{proof}

In the proof of Proposition~\ref{prop-2-8} we invoked the following.

\begin{lemma}\label{lem-2-9}
Let $Y$ be a nonsingular real algebraic variety, $U$ a nonempty Zariski open subset of $Y$, and $\tau \colon U \times \R^n \to Y$ a regular map satisfying $\tau(y,0)=y$ for all $y\in U$. Let $\varphi \colon Y \to \R$ be a regular function with $\varphi^{-1}(0) =
Y \setminus U$. Then, for each nonnegative integer $k$, there exists a positive integer $r(k)$ such that
for every integer $r \geq r(k)$ the map
$\tau^{(r)} \colon Y \times \R^n \to Y$ defined by
\begin{equation*}
\tau^{(r)}(y,w) = 
\begin{cases}
\tau(y, \varphi(y)^r w) &\text{for } (y,w) \in U \times \R^n\\
y &\text{for } (y,w) \in (Y \setminus U) \times \R^n
\end{cases}
\end{equation*}
is $k$-regulous.
\end{lemma}

\begin{proof}
We may assume that $Y$~is an algebraic subset of~$\R^m$. Then
\begin{equation*}
\tau(y,w) = (\tau_1(y,w), \ldots, \tau_m(y,w)) \quad \text{for all }
(y,w) \in U \times \R^n,
\end{equation*}
where the $\tau_i \colon U \times \R^n \to \R$ are regular functions for $i=1,\ldots,m$. Note that
\begin{equation*}
\tau_i(y,0) = y_i \quad \text{for all } y=(y_1, \ldots, y_m) \in U.
\end{equation*}
By \cite[Proposition~3.2.3]{bib4}, there exist polynomial functions $p_i, q_i \colon
\R^m \times \R^n \to \R$ such that
\begin{equation*}
q_i^{-1}(0) \cap (U \times \R^n) = \varnothing \quad \text{and} \quad
\tau_i(y,w) = \frac{p_i(y,w)}{q_i(y,w)} \quad \text{for all } (y,w) \in U \times \R^n.
\end{equation*}
We get
\begin{equation*}
\tau_i(y,w) - \tau_i(y,0) = \frac{p_i(y,w) q_i(y,0) - p_i(y,0)
q_i(y,w)}{q_i(y,w) q_i(y,0)}
\end{equation*}
and hence
\begin{equation*}
\tau_i(y,w) = y_i + \sum_{j=1}^n \tau_{ij}(y,w)w_j,
\end{equation*}
where the $\tau_{ij} \colon U \times \R^n \to \R$ are regular functions and
$w = (w_1, \ldots, w_n)$.

Let $k$ be a nonnegative integer and let $\pi \colon Y \times \R^n \to Y$ be the canonical projection. Since
$(\varphi \circ \pi)^{-1}(0) = (Y \times \R^n) \setminus (U \times \R^n)$,
there exists a positive integer~$r(k)$ such that for each integer $r
\geq r(k)$ the functions $\tau_{ij}^{(r)} \colon Y \times \R^n \to \R$
defined by
\begin{equation*}
\tau_{ij}^{(r)}(y,w) =
\begin{cases}
\tau_{ij}(y,w)\varphi(y)^r & \text{for } (y,w) \in U \times \R^n\\
0 & \text{for } (y,w) \in (Y \times \R^n) \setminus (U \times \R^n)
\end{cases}
\end{equation*}
are of class $\C^k$ for $i=1,\ldots,m$, $j=1,\ldots,n$ (see
\cite[Proposition~3.4]{bib41a}, and also \cite[Lemma~5.2]{bib17} for
$Y=\R^m$). Now we define a map $\tau^{(r)} \colon Y \times \R^n \to \R^m$ by
\begin{equation*}
\tau^{(r)}(y,w) = (\tau_1^{(r)}(y,w), \ldots, \tau_m^{(r)}(y,w)),
\end{equation*}
where
\begin{equation*}
\tau_i^{(r)}(y,w) = y_i + \sum_{j=1}^n \tau_{ij}^{(r)}(y,
\varphi(y)^rw) w_j \quad \text{for } i=1,\ldots,m.
\end{equation*}
By construction, the map $\tau^{(r)}$ is of class $\C^k$. Furthermore,
\begin{equation*}
\tau^{(r)}(y,w)=
\begin{cases}
\tau(y, \varphi(y)^r w)  &\text{for } (y,w) \in U\times \R^n\\
y &\text{for } (y,w) \in (Y \setminus U)  \times \R^n.
\end{cases}
\end{equation*}
The proof is complete because the restrictions of $\tau^{(r)}$ to $U \times \R^n$ and $(Y \setminus U) \times \R^n$ are regular maps.
\end{proof}

Further study is needed to reveal the relationship between uniform rationality and $k$-malleability.

We consider $\R^n$ endowed with the Euclidean norm $\norm{-}$. If $A$ is a nonempty subset of~$\R^n$ and $x \in \R^n$, we write $\dist(x,A)$ for the Euclidean distance from $x$ to $A$.

\begin{lemma}\label{lem-2-10}
Let $X$ be a real algebraic variety, $n$ a nonnegative integer, and $U$ a Zariski open neighborhood of $X \times \{0\}$ in $X \times \R^n$. Then there exists a regular function $\varepsilon \colon X \to \R$ such that
\begin{equation*}
    \varepsilon(x) > 0 \quad \text{and} \quad
    \left(x, \varepsilon(x) \frac{v}{1 + \norm{v}^2}\right) \in U \quad
    \text{for all }(x,v) \in X \times \R^n.
\end{equation*}
\end{lemma}

\begin{proof}
We may assume that $X$ is an algebraic subset of $\R^m$ and $U \neq X \times \R^n$. Then we choose a polynomial function $\eta \colon \R^m \times \R^n \to \R$ such that $\eta(x,v) \geq 0$ for all $(x,v) \in \R^m \times \R^n$ and the zero set of $\eta$ is the algebraic subset $Z \coloneqq (X \times \R^n) \setminus U$ of $\R^m \times \R^n$. Since the distance function
\begin{equation*}
    \R^m \times \R^n \to \R, \quad (x,v) \mapsto \dist((x,v),Z)
\end{equation*}
is a continuous semialgebraic function whose zero set is $Z$, by \cite[Theorem~2.6.6]{bib4}, there exist a positive integer $N$ and a continuous semialgebraic function $h \colon \R^m \times \R^n \to \R$ such that
\begin{equation*}
    \eta(x,v)^N = h(x,v) \dist((x,v), Z) \quad \text{for all } (x,v) \in \R^m \times \R^n.
\end{equation*}
Thus, according to \cite[Proposition~2.6.2]{bib4}, there exist a real constant $c > 0$ and a positive integer $r$ such that
\begin{equation*}
    \abs{h(x,v)} \leq c(1 + \norm{(x,v)}^2)^r \quad
    \text{for all } (x,v) \in \R^m \times \R^n.
\end{equation*}
Consequently,
\begin{equation*}
    \eta(x,0)^N \leq c(1 + \norm{x}^2)^r \dist((x,0), Z) \quad
    \text{for all } x \in X.
\end{equation*}
It follows that the function $\varepsilon \colon X \to \R$ defined by
\begin{equation*}
    \varepsilon(x) = \frac{\eta(x,0)^N}{2c(1 + \norm{x}^2)^r} \quad \text{for all } x \in X
\end{equation*}
has the required properties.
\end{proof}

\begin{proposition}\label{prop-2-11}
Let $Y$ be a malleable nonsingular real algebraic variety. Then every Zariski open subset $Y_0$ of $Y$ is a malleable variety.
\end{proposition}

\begin{proof}
By Lemma~\ref{lem-2-5}, there exists a spray $(E,p,s)$ for $Y$ such that $p \colon E = Y \times \R^n \to Y$ is the product vector bundle. Let $p_0 \colon Y_0 \times \R^n \to Y_0$ be the product vector bundle over~$Y_0$. Note that the set $U \coloneqq E_0 \cap s^{-1}(Y_0)$ is a Zariski open neighborhood of $Y_0 \times \{0\}$ in $E_0$. By Lemma~\ref{lem-2-10}, there exists a regular function $\varepsilon \colon Y_0 \to \R$ such that
\begin{equation*}
    \varepsilon(y) > 0 \quad \text{and} \quad
    \left(y, \varepsilon(y) \frac{v}{1 + \norm{v}^2}\right) \in U
    \quad \text{for all } (y,v) \in E_0.
\end{equation*}
Obviously, the map
\begin{equation*}
    s_0 \colon Y_0 \times \R^n \to Y_0, \quad
    (y,v) \mapsto s\left(y, \varepsilon(y) \frac{v}{1 + \norm{v}^2}\right)
\end{equation*}
is regular. Since the derivative of the map
\begin{equation*}
    \R^n \to \R^n, \quad v \mapsto \frac{v}{1 + \norm{v}^2}
\end{equation*}
at $0 \in \R^n$ is an isomorphism, it follows that $(E_0, p_0, s_0)$ is a dominating spray for $Y_0$. Thus $Y_0$ is a malleable variety.
\end{proof}

Proposition~\ref{prop-2-11} is a rich source of new examples of malleable real algebraic varieties.

\begin{example}\label{ex-2-12}
Let $Y$ be a real algebraic variety that is a homogeneous space for some linear real algebraic group. As recalled in Example~\ref{ex-2-6}, $Y$ is a malleable variety. Thus, by Proposition~\ref{prop-2-11}, every Zariski open subset of $Y$ is a malleable variety.
\end{example}

\begin{proposition}\label{prop-2-13}
Every uniformly rational real algebraic variety is locally malleable.
\end{proposition}

\begin{proof}
According to Definition~\ref{def-1-1}, it suffices to prove that every Zariski open subset of~$\R^n$ is a malleable variety. This follows immediately from Proposition~\ref{prop-2-11} because $\R^n$~is a malleable variety (to see that $\R^n$ is malleable, consider the map $\R^n \times \R^n \to \R^n$, ${(y,v) \mapsto y+v}$).
\end{proof}

\begin{corollary}\label{cor-2-14}
Let $k$ be a positive integer. Every uniformly rational real algebraic variety is $k$-malleable.
\end{corollary}

\begin{proof}
It suffices to combine Propositions \ref{prop-2-8}~and~\ref{prop-2-13}.
\end{proof}

\section{Sections of malleable submersions}\label{sec:3}

We will use freely terminology and notation introduced in Section~\ref{sec:2}.

\begin{notation}\label{not-3-1}
Let $X$, $Z$ be nonsingular real algebraic
varieties, and let $h \colon Z \to X$ be a regular map that is a surjective submersion. Furthermore, let $V(h)$ denote the algebraic vector
subbundle of the tangent bundle $TZ$ to $Z$ defined by
\begin{equation*}
V(h)_z = \Ker(d_z h \colon T_zZ \to T_{h(z)}X) \quad \text{for all } z
\in Z,
\end{equation*}
where $d_z h$ is the derivative of $h$ at $z$. Clearly, $V(h)_z$ is the tangent space to the fiber $h^{-1}(h(z))$.
\end{notation}

Let $k$ be a nonnegative integer or $k=\infty$, $U$ an open subset of $X$, and $\XC$ a stratification of $X$. A~map $f \colon U \to Z$ is called a
\emph{section} (over $U$) of $h \colon Z \to X$ if $h(f(x))=x$ for all $x \in U$. A~section that is a $\C^k$ map is called a \emph{$\C^k$ section}. 
By a \emph{homotopy of $\C^k$ sections} we mean a
continuous map $F \colon U \times [0,1] \to Z$ such that $F_t \colon U \to Z$, $x \mapsto F(x,t)$ is a $\C^k$ section for every $t \in
[0,1]$. Two $\C^k$ sections $f_0,f_1 \colon U \to Z$ are said to be
\emph{homotopic through $\C^k$ sections} if there exists a homotopy $F \colon U \times [0,1] \to Z$ of $\C^k$ sections with $F_0=f_0$ and
$F_1=f_1$. 
A global section $g \colon X \to Z$ that is a $k$-regulous  (resp.\ $(k,\XC)$-regular) map is called a \emph{$k$-regulous} (resp.\ \emph{$(k,\XC)$-regular}) \emph{section}. We say that a $\C^k$ section $f \colon U \to Z$
can be \emph{approximated by global $k$-regulous} (resp. \emph{global $(k,\XC)$-regular}) \emph{sections in the $\C^k$ topology}
if for every neighborhood $\U$ of $f$ in the space $\C^k(U,Z)$ of all $\C^k$ maps there exists a global $k$-regulous (resp. global $(k,\XC)$-regular) section $g \colon X \to Z$ such that $g|_U$
belongs to~$\U$. 
To study approximation by global $k$-regulous or global $(k,\XC)$-regular sections, we need several notions and auxiliary results.

\begin{definition}\label{def-3-2}
Let $h \colon Z \to X$ be the submersion of Notation~\ref{not-3-1},  and $k$ a positive integer or $k=\infty$.
\begin{iconditions}
\item\label{def-3-2-i}
A \emph{$k$-regulous spray} for $h \colon Z \to X$ is a triple
$(E,p,s)$, where $p \colon E \to Z$ is an algebraic vector bundle
over $Z$ and $s \colon E \to Z$ is a $k$-regulous map such that
\begin{equation*}
s(E_z) \subseteq h^{-1}(h(z)) \quad \text{and} \quad s(0_z)=z
\quad \text{for all } z \in Z.
\end{equation*}
\item\label{def-3-2-ii}
A $k$-regulous spray $(E,p,s)$ for $h \colon Z \to X$ is said to be \emph{dominating} if the de\-riv\-a\-tive 
    $d_{0_z}s \colon T_{0_z}E \to T_zZ$
maps the subspace $E_z = T_{0_z}E_z$ of $T_{0_z}E$ onto $V(h)_z$, that is,
\begin{equation*}
d_{0_z}s(E_z) = V(h)_z \quad \text{for all } z \in Z.
\end{equation*}

\item\label{def-3-2-iii}
The submersion $h \colon Z \to X$ is called \emph{$k$-malleable} if it admits a dominating $k$-regulous spray.
\end{iconditions}
For simplicity, $\infty$-regulous sprays, dominating $\infty$-regulous sprays and $\infty$-malleable submersions are called \emph{sprays}, \emph{dominating sprays} and \emph{malleable submersions}, respectively.
\end{definition}

Note that if $X$ is reduced to a point, then Definition~\ref{def-3-2} coincides with Definition~\ref{def-2-4}. Since $\infty$-regulous maps are regular, it follows that the concepts of spray, dominating spray, and malleable submersion in Definition~\ref{def-3-2} above are identical with those in \cite[Definition~3.2]{bib13}. Basic properties of dominating sprays for $h \colon Z \to X$ are established in \cite[Section~3]{bib13}. Taking into account all the necessary modifications, in the next lemmas we prove analogous results for $k$-regulous sprays.

\begin{lemma}\label{lem-3-3}
Let $h\colon Z \to X$ be the submersion of Notation~\ref{not-3-1},  and $k$ a positive integer or $k=\infty$. If the submersion $h\colon Z \to X$ is $k$-malleable, then it admits a
dominating $k$-regulous spray $(E,p,s)$ such that $p \colon E = Z \times \R^n \to Z$ is
the product vector bundle.
\end{lemma}

\begin{proof}
Let $(\tilde E, \tilde p, \tilde s)$ be a
dominating $k$-regulous spray for $h \colon Z \to X$. Choose a nonnegative integer $n$ and a surjective algebraic morphism $\alpha \colon E \to \tilde E$ from the product
vector bundle $p \colon E = Z \times \R^n \to Z$
onto $\tilde p \colon \tilde E \to Z$.
By Corollary~\ref{cor-2-3}, the map $s \colon E \to Z$, $s(z,v) = \tilde s(\alpha(z,v))$ is regulous, so $(E,p,s)$ is a dominating $k$-regulous spray  for $h \colon Z \to X$.
\end{proof}

\begin{lemma}\label{lem-3-4}
Let $h\colon Z \to X$ be the submersion of Notation~\ref{not-3-1}, and $k$ a~positive integer or $k=\infty$. 
Suppose that $(E,p,s)$ is a dominating $k$-regulous spray for $h \colon Z \to X$. Let $U$ be an open subset of $X$ and let $F \colon U \times [0,1] \to Z$ be a homotopy of $\C^k$ sections of $h \colon Z \to X$. Let $U_0$ be an open subset of $X$ whose closure $\overline U_0$ is compact and contained in $U$. Let $t_0$ be a point
in $[0,1]$. Then there exist a neighborhood $I_0$ of $t_0$ in $[0,1]$
and a continuous map $\xi \colon X \times I_0 \to E$ such that
\begin{inthm}[widest=3.4.2]
\item\label{lem-3-4-1} $p(\xi(x,t)) = F(x,t_0)$ for all $(x,t) \in U_0
\times I_0$,

\item\label{lem-3-4-2} $\xi(x,t_0) = 0_{F(x,t_0)}$ for all $x \in U_0$,

\item\label{lem-3-4-3} $s(\xi(x,t)) = F(x,t)$ for all $(x,t) \in U_0
\times I_0$,

\item\label{lem-3-4-4} for every $t \in I_0$ the map $U_0 \to E$, $x
\mapsto \xi(x,t)$ is of class $\C^k$.
\end{inthm}
\end{lemma}

\begin{proof}
Consider the $\C^{k-1}$ morphism
\begin{equation*}
    \alpha \colon E \to V(h), \quad v \mapsto d_{0_{p(v)}} s(v)
\end{equation*}
of algebraic vector bundles
(by convention, $\infty - 1 = \infty$). 
Since $\alpha$ is a surjective morphism, its kernel $K$ is a $\C^{k-1}$
 vector subbundle of~$E$. Hence $E$ can be written as the direct sum $E =  E' \oplus K$  for some $\C^{k-1}$ vector subbundle $E'$ of~$E$. The restriction $\alpha|_{E'} \colon E' \to V(h)$ is a $\C^{k-1}$ isomorphism of vector bundles, so we can choose a $\C^{k-1}$ morphism $\beta \colon V(h) \to E$ that induces an isomorphism of $V(h)$ onto $E'$. Let $\varphi$ be the global $\C^{k-1}$ section of the algebraic vector bundle $\Hom(V(h), E)$ that is determined by $\beta$. If $\psi$ is a $\C^k$ section of $\Hom(V(h), E)$ sufficiently close to $\varphi$ in the strong $\C^0$ topology (see \cite[p.~35]{bib27} for the definition of the strong $\C^0$ topology), then the $\C^k$ morphism $\gamma \colon V(h) \to E$ corresponding to $\psi$ is injective and the image of $\gamma$ is a $\C^k$ vector subbundle $\hat p \colon \hat E \to Z$ of $p \colon E \to Z$ such that $E = \hat E \oplus K$. By construction, the restriction $\alpha|_{\hat E} \colon \hat E \to V(h)$ is a $\C^{k-1}$ isomorphism.
 
 Let $f \colon U \to Z$ be a $\C^k$ section of $h \colon Z \to X$ defined by $f(x) = F(x,t_0)$ for all $x\in U$. Denote
by $p_f \colon \hat E_f \to U$ the pullback of the vector bundle $\hat p \colon \hat E
\to Z$ under the map $f$. Recall that $p_f \colon \hat E_f \to
U$ is a $\C^k$ vector bundle, where
\begin{equation*}
\hat E_f \coloneqq \{(x,v) \in U \times \hat E : f(x) = \hat p(v)\}, \quad p_f(x,v) = x.
\end{equation*}
Note that
\begin{equation*}
s_f \colon E_f \to Z, \quad s_f(x,v)=s(v)
\end{equation*}
is a $\C^k$ map.

Let $x \in U$. The zero vector in the fiber $(\hat E_f)_x$ is $(x,
0_{f(x)})$, where $0_{f(x)}$ is the zero vector in the fiber
$\hat E_{f(x)}$. Since $s_f(x,0_{f(x)}) = s(0_{f(x)}) = f(x)$, it follows
that $s_f$ induces a $\C^k$ diffeomorphism between the zero section $Z(\hat E_f)$ and $f(U)$.
Moreover, the derivative
\begin{equation*}
d_{(x, 0_{f(x)})} s_f \colon T_{(x, 0_{f(x)})} \hat E_f \to T_{f(x)}Z
\end{equation*}
is an isomorphism because
\begin{equation*}
d_{0_z} s|_{\hat E_z} = \alpha|_{\hat E_z} \colon \hat E_z \to V(h)_z
\end{equation*}
is an isomorphism for all $z \in Z$. Consequently, $s_f$ is a local
diffeomorphism at the point $(x,0_{f(x)})$. Thus, by \cite[(12.7)]{bib14b}, there exist an open
neighborhood $M \subseteq \hat E_f$ of the zero section $Z(\hat E_f)$ and an
open neighborhood $N \subset Z$ of $f(U)$ such that the restriction
$\sigma \colon M \to N$ of $s_f \colon \hat E_f \to Z$ is a $\C^k$ diffeomorphism.
Since $\overline U_0$ is a compact subset of $U$, we can choose an open
neighborhood~$I_0$ of $t_0$ in~$[0,1]$ such that $F_t(\overline U_0) \subseteq N$
for all $t \in I_0$. Therefore, for each $t \in I_0$, there exists a
unique $\C^k$ map $\zeta_t \colon U_0 \to \hat E_f$ satisfying
$\zeta_t(U_0) \subseteq M$ and $F_t(x) = \sigma(\zeta_t(x))$ for all $x
\in U_0$. We have $\zeta_t(x) = (\alpha_t(x),
\xi_t(x))$, where $\alpha_t \colon U_0 \to U$ and $\xi_t \colon U_0 \to
\hat E \subseteq E$ are $\C^k$ maps with
\begin{equation*}
f(\alpha_t(x)) = \hat p(\xi_t(x)) = p(\xi_t(x)) \quad \text{and} \quad s(\xi_t(x)) =
F_t(x).
\end{equation*}
By Definition~\ref{def-3-2}\ref{def-3-2-i}, $s(\xi_t(x)) \in
h^{-1}(h(p(\xi_t(x))))$, and hence
\begin{equation*}
h(s(\xi_t(x))) = h(f(\alpha_t(x))) = \alpha_t(x).
\end{equation*}
On the other hand,
\begin{equation*}
h(s(\xi_t(x))) = h(F_t(x)) = x.
\end{equation*}
Consequently, $\alpha_t(x)=x$. It follows that $\zeta_t \colon U_0 \to
\hat E_f$ is a $\C^k$ section, over $U_0$, of the vector bundle $p_f
\colon \hat E_f \to X$. Clearly, $\zeta_{t_0}(U_0) \subseteq Z(\hat E_f)$.
Furthermore, the map
\begin{equation*}
\zeta \colon U_0 \times I_0 \to \hat E_f, \quad (x,t) \mapsto \zeta_t(x)
\end{equation*}
is continuous. Note that $\zeta(x,t) = (x,\xi(x,t))$, where
\begin{equation*}
\xi \colon U_0 \times I_0 \to \hat E \subseteq E, \quad (x,t) \mapsto \xi_t(x)
\end{equation*}
is a continuous map with $p(\xi(x,t)) = f(x) = F(x,t_0)$ for all $(x,t)
\in U_0 \times I_0$. By construction, the map~$\xi$ satisfies conditions
\ref{lem-3-4-1}--\ref{lem-3-4-4}.
\end{proof}

The key lemma is the following.

\begin{lemma}\label{lem-3-5}
Assume that the submersion $h \colon Z \to X$ of Notation~\ref{not-3-1} is $k$-malleable, where $k$ is a nonnegative integer or $k=\infty$. Let $U$ be an open subset of $X$ and let ${F \colon U \times [0,1] \to Z}$ be a
homotopy of $\C^k$ sections of $h \colon Z \to X$. Let $U_0$ be an open subset of $X$ whose closure~$\overline U_0$ is compact and contained in $U$. Then there exist a dominating $k$-regulous spray
$(E,p,s)$ for $h \colon Z \to X$ and a continuous map $\xi \colon
U_0 \times [0,1] \to E$ such that $p \colon E = Z \times \R^n \to Z$
is the product vector bundle and $\xi(x,t) = (F(x,0), \eta(x,t))$ for
all $(x,t) \in U_0 \times [0,1]$, where the map $\eta \colon U_0 \times
[0,1] \to \R^n$ satisfies
\begin{inthm}[widest=3.5.3]
\item\label{lem-3-5-1} $\eta(x,0)=0$ for all $x \in U_0$,

\item\label{lem-3-5-2} $s(F(x,0),\eta(x,t)) = F(x,t)$ for all $(x,t) \in
U_0 \times [0,1]$,

\item\label{lem-3-5-3} for every $t \in [0,1]$ the map $U_0 \to \R^n$,
$x \mapsto \eta(x,t)$ is of class $\C^k$.
\end{inthm}
\end{lemma}

\begin{proof}
By Lemma~\ref{lem-3-3}, the submersion $h \colon Z \to X$ admits a
dominating $k$-regulous spray $(\tilde E, \tilde p, \tilde s)$
such that $\tilde p \colon \tilde E = Z \times \R^m \to Z$ is the product
vector bundle. In view of Lemma~\ref{lem-3-4} and the compactness of the
interval $[0,1]$ (see the Lebesgue lemma for compact metric spaces
\cite[p.~28, Lemma~9.11]{bib14a}), there exists a partition $0=t_0 < t_1
< \cdots < t_r =1$ of~$[0,1]$ such that for each $i=1,\ldots,r$ there
exists a continuous map $\xi^i \colon U_0 \times [t_{i-1},t_i] \to \tilde
E$ with the following properties:
\begin{itemize}
\item $\xi^i(x,t) = (F(x,t_{i-1}), \eta^i(x,t))$ for all $(x,t) \in U_0
\times [t_{i-1},t_i]$,

\item $\eta^i(x, t_{i-1}) = 0$ for all $x \in U_0$,

\item $\tilde s(F(x,t_{i-1}), \eta^i(x,t)) = F(x,t)$ for all $(x,t) \in
U_0 \times [t_{i-1}, t_i]$,

\item for every $t \in [t_{i-1}, t_i]$ the map $U_0 \to \R^m$, $x
\mapsto \eta^i(x,t)$ is of class $\C^k$.
\end{itemize}
For $i=1,\ldots,r$ we define recursively a dominating $k$-regulous spray
$(E^{(i)}, p^{(i)}, s^{(i)})$ for $h \colon {Z \to X}$ by
\begin{equation*}
(E^{(i)}, p^{(i)}, s^{(i)}) = (\tilde E, \tilde p,\tilde s) \quad \text{if } i=1,
\end{equation*}
while for $i \geq 2$ we require
\begin{equation*}
p^{(i)} \colon E^{(i)} = Z \times (\R^m)^i \to Z
\end{equation*}
to be the product vector bundle and set
\begin{equation*}
s^{(i)} \colon E^{(i)} \to Z, \quad s^{(i)}(z, v_1,\ldots,v_i) =
s^{(1)}(s^{(i-1)}(z, v_1,\ldots,v_{i-1}), v_i),
\end{equation*}
where $z \in Z$ and $v_1,\ldots,v_i \in \R^m$ (the map $s^{(i)}$ is $k$-regulous by Corollary~\ref{cor-2-3}).

In particular, $(E,p,s) \coloneqq (E^{(r)}, p^{(r)}, s^{(r)})$
is a dominating $k$-regulous spray for ${h \colon Z \to X}$. Note that $p
\colon E = Z \times \R^n \to Z$ is the product vector bundle with $\R^n
= (\R^m)^r$. Now, consider a map
\begin{equation*}
\xi \colon U_0 \times [0,1] \to E, \quad \xi(x,t) = (F(x,0),
\eta(x,t)),
\end{equation*}
where $\eta \colon U_0 \times [0,1] \to \R^n = (\R^m)^r$ is defined by
\begin{equation*}
\eta(x,t) = (\eta^1(x,t), 0,\ldots,0)
\end{equation*}
for all $(x,t) \in U_0 \times [t_0,t_1]$, and
\begin{equation*}
\eta(x,t) = (\eta^1(x,t_1), \ldots, \eta^{i-1}(x,t_{i-1}), \eta^i(x,t),
0,\ldots, 0)
\end{equation*}
for all $(x,t) \in U_0 \times [t_{i-1}, t_i]$ with $i=2,\ldots,r$. One
readily checks that $\eta$ is a well-defined continuous map satisfying
\ref{lem-3-5-1}--\ref{lem-3-5-3}.
\end{proof}

Now we are ready to prove the following approximation result for sections.

\begin{theorem}\label{th-3-6}
Assume that the submersion $h\colon Z \to X$ of Notation~\ref{not-3-1} is $k$-malleable, where $k$ is a nonnegative integer or $k=\infty$. Let $U$ be an open subset of $X$ and let $f \colon U \to Z$ be a $\C^k$ section
of $h \colon Z \to X$ that is homotopic through $\C^k$ sections to the restriction $f_0|_U$ of a global
$k$-regulous section $f_0 \colon X \to Z$. Then 
$f$ can be approximated by global $k$-regulous sections.
\end{theorem}

\begin{proof}
Let $F \colon U \times [0,1] \to Z$ be a homotopy of $\C^k$ sections
such that $F_0=f_0|_U$ and $F_1 = f$. Let $U_0$ be an open subset of $X$ whose closure $\overline U_0$ is compact and contained in $U$, and let $(E,p,s)$, $\xi \colon U_0 \times [0,1] \to E$, $\xi(x,t)
= (F(x,0), \eta(x,t))$, ${\eta \colon U_0 \times [0,1] \to \R^n}$ be as in
Lemma~\ref{lem-3-5}. In particular, we have
\begin{equation*}
s(f_0(x), \eta(x,1)) = s(F(x,0), \eta(x,1)) = F(x,1) = f(x) \quad \text{for all } x \in U_0.
\end{equation*}
By the Weierstrass approximation theorem, there exists a regular map $\beta \colon X \to \R^n$ such that the restriction $\beta|_{U_0}$ is arbitrarily close to the $\C^k$ map $\eta_1 \colon U_0 \to \R^n$, $x \mapsto \eta(x,1)$ in the $\C^k$ topology. Then
\begin{equation*}
    g \colon X \to Z, \quad x \mapsto s(f_0(x), \beta(x))
\end{equation*}
is a $k$-regulous map (by Corollary~\ref{cor-2-3}) such that $g|_{U_0}$ is close to $f|_{U_0}$ in the $\C^k$ topology. Moreover, in view of Definition~\ref{def-3-2}\ref{def-3-2-i}, $g \colon X \to Z$ is a section of $h \colon Z \to X$. The proof is complete because $U_0$ is chosen in an arbitrary way.
\end{proof}

We also have the following variant of Theorem~\ref{th-3-6}.

\begin{theorem}\label{th-3-7}
Assume that the submersion $h \colon Z \to X$ of Notation~\ref{not-3-1} is malleable. Let $\XC$ be a stratification of $X$, and $k$ a positive integer or $k=\infty$. Let $U$ be an open subset of $X$ and let $f \colon U \to Z$ be a $\C^k$ section of
$h \colon Z \to X$ that is homotopic through $\C^k$~sections to the restriction $f_0|_U$ of a global
$(k,\XC)$-regular section $f_0 \colon X \to Z$. Then $f$~can be approximated by global $(k,\XC)$-regular sections in the $\C^k$ topology.
\end{theorem}

\begin{proof}
Let $F \colon U \times [0,1] \to Z$ be a homotopy of $\C^k$ sections
such that $F_0=f_0|_U$ and $F_1 = f$. Let $U_0$ be an open subset of $X$ whose closure $\overline U_0$ is compact and contained in $U$, and let $(E,p,s)$, $\xi \colon U_0 \times [0,1] \to E$, $\xi(x,t)
= (F(x,0), \eta(x,t))$, ${\eta \colon U_0 \times [0,1] \to \R^n}$ be as in
Lemma~\ref{lem-3-5} (with $k=\infty$, in which case $s \colon E \to X$ is a regular map). In particular, we have
\begin{equation*}
s(f_0(x), \eta(x,1)) = s(F(x,0), \eta(x,1)) = F(x,1) = f(x) \quad \text{for all } x \in U_0.
\end{equation*}
By the Weierstrass approximation theorem, there exists a regular map $\beta \colon X \to \R^n$ such that the restriction $\beta|_{U_0}$ is arbitrarily close to the $\C^k$ map $\eta_1 \colon U_0 \to \R^n$, $x \mapsto \eta(x,1)$ in the $\C^k$ topology. Then
\begin{equation*}
    g \colon X \to Z, \quad x \mapsto s(f_0(x), \beta(x))
\end{equation*}
is a $(k, \XC)$-regular map such that $g|_{U_0}$ is close to $f|_{U_0}$ in the $\C^k$ topology. Moreover, in view of Definition~\ref{def-3-2}\ref{def-3-2-i}, $g \colon X \to Z$ is a section of $h \colon Z \to X$. The proof is complete because $U_0$~is chosen in an arbitrary way.
\end{proof}

The most important special cases of Theorems \ref{th-3-6}~and~\ref{th-3-7} are obtained by taking $U=X$.

\section{Applications}\label{sec:4}
To discuss applications of Theorems \ref{th-3-6}~and~\ref{th-3-7} we need the following observation. 

\begin{lemma}\label{lem-4-1}
Let $X$, $Y$ be nonsingular real algebraic varieties,  and $k$ a positive integer or $k=\infty$. Assume that the
variety $Y$ is $k$-malleable. Then the canonical projection
\begin{equation*}
h \colon X \times Y \to X, \quad (x,y) \mapsto x
\end{equation*}
is a $k$-malleable submersion.
\end{lemma}

\begin{proof}
Let $(E,p,s)$ be a dominating $k$-regulous spray for $Y$. We obtain a
dominating $k$-regulous spray $(\tilde E, \tilde p, \tilde s)$
for $h \colon X \times Y \to X$ setting
\begin{align*}
&\tilde E = \{ ((x,y),v) \in (X \times Y) \times E : y=p(v)\},\\
&\tilde p \colon \tilde E \to X \times Y, \quad ((x,y),v) \mapsto
(x,y),\\
&\tilde s \colon \tilde E \to X \times Y, \quad ((x,y),v) \mapsto
(x,s(v)).\qedhere
\end{align*}
\end{proof}

\begin{theorem}\label{th-4-2}
Let $X$, $Y$ be nonsingular real algebraic varieties,  and $k$ a positive integer or $k=\infty$. Assume that the variety $Y$ is
$k$-malleable. Then, for a $\C^k$ map
$f \colon X \to Y$, the following conditions are equivalent:
\begin{conditions}
\item\label{th-4-2-a} $f$ can be approximated by $k$-regulous maps in the
$\C^k$ topology.

\item\label{th-4-2-b} $f$ is homotopic to a $k$-regulous map.
\end{conditions}
\end{theorem}

\begin{proof}
The implication \ref{th-4-2-a}$\Rightarrow$\ref{th-4-2-b} holds because $X$ deformation retracts to some compact subset $K \subset X$ \cite[Corollary~9.3.7]{bib4}, and any two continuous maps from $K$ into $Y$ that are sufficiently close in the compact-open topology are homotopic (the latter assertion is valid if $Y$ is an arbitrary $\Cinfty$ manifold).

It remains to prove \ref{th-4-2-b}$\Rightarrow$\ref{th-4-2-a}. To this
end let $\Phi \colon X \times [0,1] \to Y$ be a homotopy such that
$\Phi_0$ is a $k$-regulous map and $\Phi_1=f$. We may assume that $\Phi$ is a
$\C^k$ map, see \cite[Proposition~10.22 and its proof]{bib48}. By
Lemma~\ref{lem-4-1}, the canonical projection $h \colon X \times Y \to
X$ is $k$-malleable. Since
\begin{equation*}
F \colon X \times [0,1] \to X \times Y, \quad (x,t) \mapsto (x,
\Phi(x,t))
\end{equation*}
is a homotopy of $\C^k$ sections of $h \colon X \times Y \to X$, it follows from Theorem~\ref{th-3-6} that the $\C^k$ section
\begin{equation*}
X \to X \times Y, \quad x \mapsto (x,f(x))
\end{equation*}
can be approximated by $k$-regulous sections in the $\C^k$ topology, hence \ref{th-4-2-a} holds.
\end{proof}

\begin{proof}[Proof of Theorem~\ref{th-1-2}]
By Corollary~\ref{cor-2-14}, the variety $Y$ is $k$-malleable, so it suffices to apply Theorem~\ref{th-4-2}.
\end{proof}

Theorem~\ref{th-4-2} not only implies Theorem~\ref{th-1-2}, but is in fact more general. Indeed, as recalled in Example~\ref{ex-2-6}, if $G$ is a linear real algebraic group, then any good $G$-space $Y$ is a malleable variety. However, it may be the case that $Y$ is not rational. This latter fact was communicated to me independently by Olivier Benoist and Olivier Wittenberg.

\begin{theorem}\label{th-4-3}
Let $X$, $Y$ be nonsingular real algebraic varieties, $\XC$ a stratification of $X$, and $k$ a positive integer or $k=\infty$. Assume that the variety
$Y$ is malleable. Then, for a
$\C^k$ map $f \colon X \to Y$, the following conditions are
equivalent:
\begin{conditions}
\item\label{th-4-3-a} $f$ can be approximated by $(k,\XC)$-regular maps.

\item\label{th-4-3-b} $f$ is homotopic to a $(k,\XC)$-regular map.
\end{conditions}
\end{theorem}

\begin{proof}
We proceed as in the proof of Theorem~\ref{th-4-2}, using Theorem~\ref{th-3-7} instead of Theorem~\ref{th-3-6}.
\end{proof}

As recalled in Example~\ref{ex-2-6}, unit spheres are malleable varieties, so Theorem~\ref{th-4-3} implies immediately the following.

\begin{corollary}\label{cor-4-4}
Let $X$ be a nonsingular real algebraic variety, $\XC$ a stratification of $X$, $p$~a~nonnegative integer, and $k$ a positive integer or $k=\infty$.  Then, for a
$\C^k$ map ${f \colon X \to \SB^p}$, the following conditions are
equivalent:
\begin{conditions}
\item\label{cor-4-4-a} $f$ can be approximated by $(k,\XC)$-regular maps.

\item\label{cor-4-4-b} $f$ is homotopic to a $(k,\XC)$-regular map.\qed
\end{conditions}
\end{corollary}

Next we prove our results on approximation by nice $k$-regulous maps
announced in Section~\ref{sec:1}.

\begin{proof}[Proof of Theorem~\ref{th-1-6}]
\ref{th-1-6-a}$\Rightarrow$\ref{th-1-6-b}. It is sufficient to prove
that any nice $k$-regulous map $g \colon X \to \SB^p$ can be
approximated by adapted $\Cinfty$ maps in the $\C^k$ topology. By Sard's
theorem, there exists a regular value $y \in \SB^p \setminus g(P(g))$
for the map $g|_{X \setminus P(g)} \colon X \setminus P(g) \to \SB^p$.
Using partition of unity and radial projection $\R^{p+1} \setminus \{0\}
\to \SB^p$, we readily construct a $\Cinfty$ map $\tilde g \colon X \to
\SB^p$, arbitrarily close to~$g$ in the $\C^k$ topology, such that
$\tilde g^{-1}(y) = g^{-1}(y)$ and $\tilde g = g$ in a neighborhood of
$g^{-1}(y)$. By construction, $\tilde g^{-1}(y)$ is a nonsingular
Zariski locally closed subset of~$X$, so the map~$\tilde g$ is adapted.

\ref{th-1-6-b}$\Rightarrow$\ref{th-1-6-a}. Our argument
is based on Corollary~\ref{cor-4-4}, so to handle the case $k=0$ we choose
an integer $l > k$. We may assume without loss of generality that the $\Cinfty$ map~$f$ is
adapted. Let $y_0\in \SB^p$ be a regular value for~$f$ such that
$f^{-1}(y_0)$ is a nonsingular Zariski locally closed subset of~$X$. By
\cite[Theorems~2.4 and 2.5]{bib34}, there exists a nice $l$-regulous map
$\varphi \colon X \to \SB^p$, homotopic to~$f$, such that
$\varphi^{-1}(y_0) = f^{-1}(y_0)$ and $\varphi(P(\varphi)) \subseteq
\{-y_0\}$. Clearly,
\begin{equation*}
f(P(\varphi)) \subset \SB^p \setminus \{y_0\}.
\end{equation*}
Using Lemma~\ref{lem-2-2}\ref{lem-2-2-ii}, we choose a stratification~$\XC$
of~$X$ such that $X \setminus P(\varphi)$ is a stratum in~$\XC$ and the
map~$\varphi$ is $(l,\XC)$-regular. By Corollary~\ref{cor-4-4}, $f$
can be approximated by $(l,\XC)$-regular maps in the $\C^l$
topology. Consequently, $f$~can be approximated by $(k,\XC)$-regular
maps in the $\C^k$ topology. If $\tilde f \colon X \to \SB^p$ is a
$(k,\XC)$-regular map close to~$f$ in the $\C^k$ topology, then $\tilde
f(P(\varphi)) \subset \SB^p \setminus \{y_0\}$. Since $P(\tilde f)
\subseteq P(\varphi)$, the map $\tilde f$ is $k$-regulous and nice,
hence \ref{th-1-6-a} holds.

\ref{th-1-6-c}$\Rightarrow$\ref{th-1-6-b}. For the proof we may assume
that the map~$f$ is weakly adapted. Let $z_0 \in \SB^p$ be a regular
value for~$f$ such that the $\Cinfty$ submanifold $f^{-1}(z_0)$ of~$X$
admits a weak algebraic approximation. It follows that $f^{-1}(z_0)$ is
isotopic in~$X$, via an arbitrarily small $\Cinfty$ isotopy, to a
nonsingular Zariski locally closed subset~$Z$ of~$X$. Such an isotopy
can be extended to a $\Cinfty$ ambient isotopy of~$X$, close to the
identity map of~$X$ in the space $\Cinfty(X,X)$, see \cite[pp.~179,
180]{bib27}. Thus there exists a $\Cinfty$ diffeomorphism $\sigma \colon
X \to X$ such that $\sigma(f^{-1}(z_0)) = Z$ and the composite map $f
\circ \sigma^{-1}$ is close to~$f$ in the $\C^k$ topology. By
construction, $z_0$ is a regular value for the $\Cinfty$ map $f \circ
\sigma^{-1}$, and $(f \circ \sigma^{-1})^{-1}(z_0) = Z$. Consequently,
the map $f \circ \sigma^{-1}$ is adapted, as required.

The proof is complete since the implication
\ref{th-1-6-b}$\Rightarrow$\ref{th-1-6-c} is obvious.
\end{proof}

\begin{proof}[Proof of Theorem~\ref{th-1-9}]
\ref{th-1-9-d}$\Rightarrow$\ref{th-1-9-a}. Let $M$ be a compact
$\Cinfty$ submanifold of~$X$, with 
\begin{equation*}
2 \dim M + 1 \leq \dim X,
\end{equation*}
such that
the unoriented bordism class of the inclusion map $i \colon M
\hookrightarrow X$ is algebraic, that is, representable by a regular map
from a compact nonsingular real algebraic variety into~$X$. By Benoist's
theorem \cite[Theorem~2.7]{bib2}, $M$~admits an algebraic approximation
in~$X$, that is, for every neighborhood~$\U$ of~$i$ in the space
$\Cinfty(M,X)$ there exists a $\Cinfty$ embedding $e \colon M \to X$
in~$\U$ such that $e(M)$ is a nonsingular Zariski closed subset of~$X$
(in particular, $M$~admits a weak algebraic approximation in~$X$).

Now suppose that \ref{th-1-9-d} holds. By Sard's theorem, there exists a
regular value $y\in \SB^p$ for~$f$. It is well-known that the
$\Z/2$-homology class represented by the $\Cinfty$ submanifold
$f^{-1}(y)$ of~$X$ is Poincar\'e dual to the cohomology class
$f^*(\sigma_p) \in H^p(X;\Z/2)$, see \cite[Proposition~2.15]{bib14}.
Therefore, by \cite[Lemma~2.3]{bib43}, the unoriented bordism class of
the inclusion map ${f^{-1}(y) \hookrightarrow X}$ is algebraic.
Consequently, by the Benoist theorem mentioned above, the $\Cinfty$
submanifold $f^{-1}(y)$ admits an algebraic approximation in~$X$. Thus,
the $\Cinfty$ map~$f$ is weakly adapted, so, in view of
Theorem~\ref{th-1-6}, condition~\ref{th-1-9-a} holds.

\ref{th-1-9-c}$\Rightarrow$\ref{th-1-9-d}. Let $g \colon X \to \SB^p$ be
a nice regulous map homotopic to~$f$. Choose a regular value $z \in
\SB^p \setminus g(P(g))$ of the map $g|_{X \setminus P(g)} \colon X
\setminus P(g) \to \SB^p$. Clearly, the compact $\Cinfty$ submanifold
$g^{-1}(z)$ of~$X$ is a nonsingular Zariski locally closed subset. The
$\Z/2$-homology class represented by $g^{-1}(z)$ is Poincar\'e dual to
the cohomology class $g^*(\sigma_p) \in H^p(X; \Z/2)$. We have
$f^*(\sigma_p) = g^*(\sigma_p)$, the maps~$f,g$ being homotopic. It
follows that the cohomology class $f^*(\sigma_p)$ is adapted, and hence
\ref{th-1-9-d} holds.

The proof is complete since the implications
\ref{th-1-9-a}$\Rightarrow$\ref{th-1-9-b} and
\ref{th-1-9-b}$\Rightarrow$\ref{th-1-9-c} are obvious.
\end{proof}

\begin{acknowledgements}
The author was partially supported by the National
Science Center (Poland) under grant number 2018/31/B/ST1/01059.
\end{acknowledgements}

\phantomsection

\end{document}